\documentclass[12pt]{article}

\usepackage{graphics}
\usepackage{epstopdf}
\usepackage{hyperref}
\hypersetup{
  colorlinks   = true,    
  urlcolor     = blue,    
  linkcolor    = blue,    
  citecolor    = red      
}
\usepackage[caption=false]{subfig}

\usepackage{amsmath}
\usepackage{csquotes}
\usepackage{tikz}
\usepackage{upgreek}
\usepackage{amssymb}
\usepackage{amsthm}
\usepackage{amsfonts}
\numberwithin{equation}{section}
\usepackage{booktabs,caption}
\usepackage[flushleft]{threeparttable}
\usepackage{stmaryrd}
\usepackage{wasysym}
\usepackage{makecell}
\usepackage{float}

\newtheorem{theorem}{Theorem}[section]
\newtheorem{definition}[theorem]{Definition}

\newtheorem{lemma}[theorem]{Lemma}

\newtheorem{corollary}[theorem]{Corollary}
\newtheorem{notation}[theorem]{Notation}
\newtheorem{remark}[theorem]{Remark}

\theoremstyle{definition}

\usepackage{nccmath}
\usepackage{mathtools}
\usepackage[reftex]{theoremref}
\usepackage[left=1in, right=1in, top=1in, bottom=1in]{geometry}

\DeclareMathOperator{\der}{Der}
\DeclareMathOperator{\inn}{inn}
\DeclareMathOperator{\ch}{char}

\newcommand{\FF }{\mathbb F}

\begin{document}


\title{Derivations of Non-Commutative Group Algebras}

\author{Praveen Manju and Rajendra Kumar Sharma}
\date{}
\maketitle

\begin{center}
\noindent{\small Department of Mathematics, \\Indian Institute of Technology Delhi, \\ Hauz Khas, New Delhi-110016, India$^{1}$}
\end{center}

\footnotetext[1]{{\em E-mail addresses:} \url{praveenmanjuiitd@gmail.com}(Corresponding Author: Praveen Manju), \url{rksharmaiitd@gmail.com}(Rajendra Kumar Sharma).}

\medskip

\begin{abstract}
In this article, we study the derivations of group algebras of some important groups, namely, dihedral ($D_{2n}$), Dicyclic ($T_{4n}$) and Semi-dihedral ($SD_{8n}$). First, we explicitly classify all inner derivations of a group algebra $\FF G$ of a finite group $G$ over an arbitrary field $\FF $. Then we classify all $\FF $-derivations of the group algebras $\FF D_{2n}$, $\FF T_{4n}$ and $\FF (SD_{8n})$ when $\FF $ is a field of characteristic $0$ or an odd rational prime $p$ by giving the dimension and an explicit basis of these derivation algebras. We explicitly describe all inner derivations of these group algebras over an arbitrary field. Finally, we classify all derivations of the above group algebras when $\FF $ is an algebraic extension of a prime field. 
\end{abstract}



\section{Introduction}
The notion of derivations, introduced from analytic theory is old and plays a significant role in the research of structure and property in algebraic systems. In this article, we consider the pure algebraic structure, namely, the group ring and its derivations which have numerous applications. For a history of group rings, we refer the reader to {\cite[Chapter 3]{CPM2002}} and for a history of derivations, we refer the reader to the survey articles  \cite{Haetinger2011}, \cite{MohammadAshraf2006}. Derivations have been highly studied in prime and semiprime rings and have been principally used in solving functional equations \cite{Brear1992}. Derivations of group rings have various applications in coding theory \cite{Creedon2019}, \cite{Boucher2014}. The explicit description of derivations is useful in the construction of codes (see \cite{Creedon2019}).

Group algebras have a rich structure and have been studied by many mathematicians. Describing the derivation algebra consisting of the derivations of group algebra is a well-known problem. The derivation problem for group rings asks if all the derivations in a group ring are inner or if the space of outer derivations is trivial. We refer the reader to \cite{AleksandrAlekseev2020}, \cite{A.A.Arutyunov2020}, \cite{Arutyunov2020b}, \cite{Arutyunov2020} for the history and importance of the derivation problem.

There has not been much research on derivations of group rings (defined purely algebraically). The study of derivations of group rings began with the paper \cite{Smith1978} where the author studies derivations of group rings of a finitely-generated, torsion-free, nilpotent group over a field. For instance, it is shown that such group rings always contain an outer derivation. In \cite{Spiegel1994}, it is proved that every derivation of an integral group ring of a finite group is inner. In \cite{MiguelFerrero1995}, the authors proved that if $G$ is a torsion group whose center $Z(G)$ has finite index in $G$ and $R$ is a semiprime ring such that $\ch(R)$ is either $0$ or does not divide the order of any element of $G$, then every $R$-derivation of the group ring $RG$ is inner. A. A. Arutyunov in his several papers (\cite{AleksandrAlekseev2020}, \cite{OrestD.Artemovych2020}, \cite{Arutyunov2021}, \cite{A.A.Arutyunov2020}, \cite{Arutyunov2023}, \cite{Arutyunov2020a}, \cite{Arutyunov2020}) studies derivations using topology and characters. In \cite{Creedon2019}, the authors study the derivations of group rings over a commutative unital ring $R$ in terms of the generators and relators of the group. They give a necessary and sufficient condition under which a map from the generating set of the group $G$ to the group ring $RG$ can be extended to a derivation of $RG$. As an application of this characterization, the authors also classify the derivations of commutative group algebras over a field of prime characteristic $p$ and that of dihedral group algebras over a field of characteristic $2$. In this article, we classify the derivation algebras of the groups algebras of dihedral, dicyclic, and semi-dihedral groups over a field of characteristic $0$ or an odd prime $p$. Thus, the derivation algebras of some important non-commutative group algebras are classified and the derivation problem for these group algebras is solved over fields of characteristic $0$ or odd prime $p$. The structure of the manuscript is as follows.

The article is divided into five sections. In Section \ref{section 2}, we state some preliminaries. In Subsection \ref{subsection 2.1}, we state some basic facts and definitions, notations, and remarks. In Subsection \ref{subsection 2.2}, we include some important results that are used in the following sections. In Section \ref{section 3}, we prove our main  \th\ref{theorem 3.3}, in which we classify all $\FF $-derivations of the dihedral group algebra $\FF D_{2n}$, where $D_{2n} = \langle a, b \mid a^{n} = 1, b^{2} = 1, (ab)^{2} = 1 \rangle$ ($n \geq 3$) and $\FF $ is a field of characteristic $0$ or an odd rational prime $p$. As a consequence, we determine the derivation algebra $\der(\FF G)$ of $\FF D_{2n}$ when $\FF $ is an algebraic extension of a prime field, and then further classify the derivations in $\der(\FF G)$ as inner or outer thus solving the derivation problem for $\FF D_{2n}$. It is established that when $\FF $ is an algebraic extension of a prime field with characteristic $0$ or an odd rational prime $p$, then all derivations of $\FF D_{2n}$ are inner provided $\ch(\FF )$ is either $0$ or an odd prime $p$ with $\gcd(n,p)=1$ and $\FF D_{2n}$ has non-zero outer derivations only when $\ch(\FF ) = p$ with $\gcd(n,p) \neq 1$. In Section \ref{section 4}, we study the derivations of dicyclic group algebras $\FF T_{4n}$, where $T_{4n} = \langle a, b \mid a^{2n}=1, a^{n}=b^{2}, b^{-1}ab = a^{-1} \rangle$ and obtain results analogous to that of Section \ref{section 3}. In Section \ref{section 5}, we study the derivations of semi-dihedral group algebras $\FF (SD_{8n})$, where $SD_{8n} = \langle a, b \mid a^{4n} = b^{2} = 1, bab = a^{2n-1}\rangle$ and obtain results analogous to that of Sections \ref{section 3} and \ref{section 4}.

\section{Preliminaries}\label{section 2}
\subsection{Basic Facts}\label{subsection 2.1}
Below we state some basic facts. Let $R$ denote a ring. 

\begin{definition}\label{definition 2.1}
A map $d:R \rightarrow R$ that satisfies $d(a + b) = d(a) + d(b)$ and $d(ab) = d(a) b + a d(b)$ for all $a, b \in R$, is called a derivation of $R$.
\end{definition}

\begin{definition}\label{definition 2.2}
A derivation $d:R \rightarrow R$ is called inner if there exists some $b \in R$ such that $d(a) = ab - ba$ for all $a \in R$, and then $d$ is denoted by $d = d_{b}$.
\end{definition} 

\begin{definition}\label{definition 2.3}
A derivation $d:R \rightarrow R$ is called outer if it is not inner.
\end{definition}

\begin{definition}\label{definition 2.4}
If $S$ is a subring of $R$, then a derivation $d: R \rightarrow R$ is called an $S$-derivation if $d(s) = 0$ for all $s \in S$. 
\end{definition}

\begin{notation}\label{notation 2.5}
We denote the set of all derivations of $R$ by $\der(R)$, the set of all inner derivations of $R$ by $\der_{\inn}(R)$ and the set of all $S$-derivations of $R$ by $\der_{S}(R)$. Also, $\ch(\FF )$ denotes the characteristic of field $\FF $ and $\gcd(n,p)$ denotes the greatest common divisor of $n$ and $p$.
\end{notation}

\begin{remark}\label{remark 2.6}
If $R$ is a ring with unity $1$ and $d \in \der(R)$, then $d(1) = 0$. If $R$ is commutative, then $\der(R)$ becomes an $R$-module with respect to the componentwise sum and module action, and $\der_{\inn}(R)$ and $\der_{S}(R)$ become its $R$-submodules. In particular, if $R = \mathcal{A}$ is an algebra over a field $\FF $, then $\der(\mathcal{A})$ becomes an $\FF $-module as well, and $\der_{\inn}(\mathcal{A})$ and $\der_{\FF }(\mathcal{A})$ (for $S = \FF $) become its $\FF $-submodules.
\end{remark}

\begin{definition}\label{definition 2.7}
If $R$ is a ring and $G$ is a group, then the group ring of $G$ over $R$ is defined as the set $$RG = \{\sum_{g \in G} a_{g} g \mid a_{g} \in R, \forall g \in G \hspace{0.2cm} \text{and} \hspace{0.2cm} |\text{supp}(\alpha)| < \infty \},$$ where for $\alpha = \sum_{g \in G} a_{g} g$, $\text{supp}(\alpha)$ denotes the support of $\alpha$ that consists of elements from $G$ that appear in the expression of $\alpha$. $RG$ is a ring concerning the componentwise addition and multiplication defined respectively by: For $\alpha = \sum_{g \in G} a_{g} g$, $\beta = \sum_{g \in G} b_{g} g$ in $RG$, $$(\sum_{g \in G} a_{g} g ) + (\sum_{g \in G} b_{g} g) = \sum_{g \in G}(a_{g} + b_{g}) g \hspace{0.2cm} \text{and} \hspace{0.2cm} \alpha \beta = \sum_{g, h \in G} a_{g} b_{h} gh.$$ If the ring $R$ is commutative having unity $1$ and the group $G$ is abelian having identity $e$, then $RG$ becomes a commutative unital algebra over $R$ with identity $1 = 1e$. We adopt the convention that empty sums are $0$ and empty products are $1$.
\end{definition}

\begin{definition}\label{definition 2.8}
For a commutative unital ring $R$ and a group $G$, the homomorphism $\varepsilon:RG \rightarrow R$ defined by $\varepsilon(\sum_{g \in G}\lambda_{g}g) = \sum_{g \in G} \lambda_{g}$ is called the augmentation mapping of $RG$ and its kernel $\Delta(G) = \{\sum_{g \in G}\lambda_{g}g \in RG \mid \sum_{g \in G} \lambda_{g} = 0\}$ is called the augmentation ideal of $RG$.
\end{definition}

\begin{definition}\label{definition 2.9}
For a group $G$, a commutative ring $R$ and a finite conjugacy class $C$ of $G$, the sum $\hat{C} = \sum_{x \in C} x$ is called a class sum of $G$ over $R$.
\end{definition}

\begin{lemma}[{\cite[Theorem 3.6.2]{CPM2002}}]\th\label{lemma 2.10}
Let $G$ be a group, $R$ be a commutative ring and $\{C_{i}\}_{i \in I}$ be the set of all conjugacy classes of $G$ which contain only a finite number of elements. Then the set $\{\hat{C}_{i}\}$ of all class sums forms a basis of $Z(RG)$, the center of $RG$.
\end{lemma}

\subsection{Some useful results}\label{subsection 2.2}
Let $R$ be a commutative unital ring and $G = \langle X \mid Y \rangle$ be a group with $X$ as its set of generators and $Y$ as the set of relators. In \cite{Creedon2019}, the authors have studied the derivations of the group ring $RG$ in terms of the generators and relators of the group $G$. They have given a necessary and sufficient condition under which a map $f:X \rightarrow RG$ can be extended to a derivation of $RG$:
\begin{theorem}[{\cite[Theorem 2.5]{Creedon2019}}]\th\label{theorem 2.11}
Let $G = \langle X \mid Y \rangle$ be a group with $X$ as its set of generators and $Y$ the set of relators. Let $F_{X}$ denote the free group on $X$ and $\phi: F_{X} \rightarrow G$ the natural onto homomorphism. Let $R$ be a commutative ring with unity and $f:X \rightarrow RG$ be a map. Then 
\begin{enumerate}
\item[(i)] $f$ can be uniquely extended to a map $\tilde{f}:F_{X} \rightarrow RG$ satisfying \begin{equation}\label{eq 2.1}
\tilde{f}(vw) = \tilde{f}(v) \phi(w) + \phi(v) \tilde{f}(w), \hspace{0.2cm} \forall \hspace{0.2cm} v, w \in F_{X}.
\end{equation}

\item[(ii)] $f$ can be uniquely extended to an $R$-derivation of $RG$ if and only if $\tilde{f}(y) = 0$ for all $y \in Y$.
\end{enumerate}
\end{theorem}
\noindent The map $\tilde{f}:F_{X} \rightarrow RG$ is defined as \begin{equation}\label{eq 2.2} \tilde{f}(x) = \begin{cases} 
f(x) & \hspace{0.2cm} \text{if $x \in X$} \\
-xf(x^{-1})x & \hspace{0.2cm} \text{if $x \in X^{-1}$} \\
0 & \hspace{0.2cm} \text{if $x = 1$}
\end{cases}\end{equation} and then on the whole of $F_{X}$ as \begin{equation}\label{eq 2.3} \tilde{f}(v) = \sum_{i=1}^{k} \left(\prod_{j=1}^{i-1} x_{j} \right) \tilde{f}(x_{i}) \left(\prod_{j=i+1}^{k} x_{j} \right),\end{equation} if $v = \prod_{i=1}^{k} x_{i}$. 
Any derivation $d$ of a group algebra $RG$ is uniquely determined by the image set $\{d(x) \mid x \in X\}$. If the generating set $X$ is finite, say, $X = \{x_{1}, ..., x_{t}\}$, then we denote such a derivation by $d_{(\bar{x_{1}}, ..., \bar{x_{t}})}$, where $\bar{x_{i}} = d(x_{i})$ ($1 \leq i \leq t$). The following theorem guarantees that for a $K$-algebra $\mathcal{A}$ where $K$ is an algebraic extension of a prime field $\FF $, $\der(\mathcal{A}) = \der_{K}(\mathcal{A})$:
\begin{theorem}
[{\cite[Theorem 2.2]{Creedon2019}}]\th\label{theorem 2.12}
Let $\mathcal{A}$ be a $K$-algebra where $K$ is an algebraic extension of a prime field $\FF $ and let $d \in \der(\mathcal{A})$. Then $d(K) = \{0\}$ and $d$ is a $K$-linear map.
\end{theorem}

The following definition generalizes the notion of augmentation ideal of a group ring.

\begin{definition}\label{definition 2.13}
For a subgroup $H$ of $G$, define $$\Delta '(H) = \{\sum_{g \in G} \lambda_{g} g \in RG \mid \sum_{g \in H} \lambda_{g} = 0 \hspace{0.1cm} \text{\&} \hspace{0.1cm} \sum_{g \in G \setminus H} \lambda_{g} = 0\}.$$ $\Delta '(H)$ is an $R$-submodule of the $R$-module $RG$ and $\Delta '(G) = \Delta(G)$. 
\end{definition}

\begin{definition}\label{definition 2.14}
We call the set $$\bar{C}(\beta) = \{\alpha \in \FF G \mid \alpha \beta = - \beta \alpha\}$$  as the anti-centralizer of $\beta$ in $RG$ and it is an $R$-submodule of the $R$-module $RG$.
\end{definition}

For an abelian group $G$ and a field $\FF $, the group algebra $\FF G$ has no non-zero inner derivations. The theorem below classifies all inner derivations of non-commutative group algebras.

\begin{theorem}\th\label{theorem 2.15}
Let $\FF $ be an arbitrary field of any characteristic and $G$ be a finite group of order $n$ having $r$ conjugacy classes. Then the dimension of $\der_{\inn}(\FF G)$ is $n-r$.
If $C_{1}, ..., C_{r}$ are the all distinct conjugacy classes of $G$ with representatives $x_{1}, ..., x_{r}$ respectively and $s$ is a positive integer such that $|C_{i}| = 1$ for all $i \in \{1, ..., s\}$ and $|C_{i}| \geq 2$ for all $i \in \{s+1, ..., r\}$. Then the set $$\mathcal{B}_{0} = \{d_{g} \mid g \in \bigcup_{i=s+1}^{r} (C_{i} \setminus \{x_{i}\})\}$$ forms an $\FF $-basis of $\der_{\inn}(\FF G)$.
\end{theorem}
\begin{proof}
$\text{dim}(\FF G) = n$. By \th\ref{lemma 2.10}, the dimension of $Z(\FF G) = \{z \in \FF G \mid z \alpha = \alpha z, \hspace{0.1cm} \forall \hspace{0.1cm} \alpha \in \FF G\}$ is $r$ and the set $\mathcal{B}_{Z} = \{\hat{C}_{i} \mid 1 \leq i \leq r\}$ of all class sums of $G$ over $\FF $ forms a basis of $Z(\FF G)$. $\mathcal{B}_{Z}$ can be extended to a basis of $\FF G$, say, $\mathcal{B}' = \{\hat{C}_{1}, ..., \hat{C}_{r}, \beta_{r+1}, ..., \beta_{n}\}$. Put $\mathcal{B} = \{d_{\beta_{i}} \mid r+1 \leq i \leq n\}$. We prove that $\mathcal{B}$ is a basis of $\der_{\inn}(\FF G)$ over $\FF $.
Let $\lambda \in \FF $ and $\beta, \beta_{1}, \beta_{2} \in \FF G$. Then for all $\alpha \in \FF G$, $d_{\lambda \beta}(\alpha) = \alpha (\lambda \beta) - (\lambda \beta) \alpha = \lambda (\alpha \beta - \beta \alpha) = (\lambda d_{\beta})(\alpha)$ and $(d_{\beta_{1}} + d_{\beta_{2}})(\alpha) = d_{\beta_{1}}(\alpha) + d_{\beta_{2}}(\alpha) = (\alpha \beta_{1} - \beta_{1} \alpha) + (\alpha \beta_{2} - \beta_{2} \alpha) = \alpha(\beta_{1} + \beta_{2}) - (\beta_{1} + \beta_{2})\alpha = d_{\beta_{1} + \beta_{2}}(\alpha)$. So $d_{\lambda \beta} = \lambda d_{\beta}$ and $d_{\beta_{1} + \beta_{2}} = d_{\beta_{1}} + d_{\beta_{2}}$.

Now let $d \in \der_{\inn}(\FF G)$. Then $d = d_{\beta}$ for some $\beta \in \FF G$. Since $\mathcal{B}'$ is a basis of $\FF G$, therefore, $\beta = \sum_{i=1}^{r} \lambda_{i} \hat{C}_{i} + \sum_{j=r+1}^{n} \mu_{j} \beta_{j}$ for some $\lambda_{i}, \mu_{j} \in \FF $ ($1 \leq i \leq r$, $r+1 \leq j \leq n$). Then 

\begin{eqnarray*}
d_{\beta} & = &  \sum_{i=1}^{r} \lambda_{i} d_{\hat{C}_{i}} + \sum_{j=r+1}^{n} \mu_{j} d_{\beta_{j}}.
\end{eqnarray*}
$\hat{C}_{i} \in Z(\FF G)$ ($1 \leq i \leq r$), so $d_{\hat{C}_{i}}(\alpha) = \alpha \hat{C}_{i} - \hat{C}_{i} \alpha = 0$ for all $\alpha \in \FF G$ so that $d_{\hat{C}_{i}} = 0, \hspace{0.1cm} \forall \hspace{0.1cm} i \in \{1, ..., r\}$. Therefore, $$d_{\beta} = \sum_{j=r+1}^{n} \mu_{j} d_{\beta_{j}}.$$
Therefore, $\mathcal{B}$ spans $\der_{\inn}(\FF G)$ over $\FF $. Now let $\lambda_{i} \in \FF $ ($r+1 \leq i \leq n$) such that $\sum_{i=r+1}^{n} \lambda_{i} d_{\beta_{i}} = 0$. 

$\Rightarrow d_{\sum_{i=r+1}^{n} \lambda_{i} \beta_{i}}(\alpha) = 0, \hspace{0.1cm} \forall \hspace{0.1cm} \alpha \in \FF G$.

$\Rightarrow \alpha \left(\sum_{i=r+1}^{n} \lambda_{i} \beta_{i}\right) - \left(\sum_{i=r+1}^{n} \lambda_{i} \beta_{i}\right)\alpha = 0, \hspace{0.1cm} \forall \hspace{0.1cm} \alpha \in \FF G$.

$\Rightarrow \sum_{i=r+1}^{n} \lambda_{i} \beta_{i} \in Z(\FF G), \hspace{0.1cm} \forall \hspace{0.1cm} \alpha \in \FF G$.

$\Rightarrow \sum_{i=r+1}^{n} \lambda_{i} \beta_{i} = \sum_{j=1}^{r} \mu_{j} \hat{C}_{j}$ for some $\mu_{j} \in \FF $ ($1 \leq j \leq r$) as $\mathcal{B}_{Z}$ is a basis of $Z(\FF G)$.

$\Rightarrow \mu_{j} = 0, \hspace{0.1cm} \forall \hspace{0.1cm} j \in \{1, ..., r\}$ and $\lambda_{i}=0, \hspace{0.1cm} \forall \hspace{0.1cm} i \in \{r+1, ..., n\}$, since $\mathcal{B}'$ being a basis is an $\FF $-linearly independent set. 

\noindent Therefore, the set $\mathcal{B}$ is $\FF $-linearly independent. Hence, $\mathcal{B}$ forms a basis of $\der_{\inn}(\FF G)$ over $\FF $. Since $|\mathcal{B}| = n-r$, therefore, $\text{dim}(\der_{\inn}(\FF G)) = n-r$. Observe that this fact follows more directly from the fact that the map $\theta:\FF G \rightarrow \der_{\inn}(\FF G)$ defined by $\theta(\alpha) = d_{\alpha}$ ($\alpha \in \FF G$) is an $\FF $-linear map with null space $Z(\FF G)$.

Suppose that $\sum_{g \in \bigcup_{i=s+1}^{r} C_{i} \setminus \{x_{i}\}} \lambda_{g} d_{g} = 0$ for $\lambda_{g} \in \FF $ ($g \in \bigcup_{i=s+1}^{r} C_{i} \setminus \{x_{i}\}$). 

$\Rightarrow d_{\sum_{g \in \bigcup_{i=s+1}^{r} C_{i} \setminus \{x_{i}\}} \lambda_{g} g} = 0$.

$\Rightarrow \alpha \left(\sum_{g \in \bigcup_{i=s+1}^{r} C_{i} \setminus \{x_{i}\}} \lambda_{g} g\right) = \left(\sum_{g \in \bigcup_{i=s+1}^{r} C_{i} \setminus \{x_{i}\}} \lambda_{g} g\right)\alpha, \hspace{0.1cm} \forall \hspace{0.1cm} \alpha \in \FF G$.

$\Rightarrow \sum_{g \in \bigcup_{i=s+1}^{r} C_{i} \setminus \{x_{i}\}} \lambda_{g} g \in Z(\FF G)$.

$\Rightarrow \sum_{g \in \bigcup_{i=s+1}^{r} C_{i} \setminus \{x_{i}\}} \lambda_{g} g = \sum_{i=1}^{r} \mu_{i} \hat{C}_{i}$ for some $\mu_{i} \in \FF $ ($1 \leq i \leq r$).

Therefore,
\begin{eqnarray*}
\sum_{i=s+1}^{r} \left(\sum_{g \in C_{i} \setminus \{x_{i}\}} \lambda_{g} g\right) & = & \sum_{i=1}^{r} \mu_{i} \hat{C}_{i} = \sum_{i=1}^{s} \mu_{i} \hat{C}_{i} + \sum_{i=s+1}^{r} \mu_{i} \hat{C}_{i} = \sum_{i=1}^{s} \mu_{i} x_{i} + \sum_{i=s+1}^{r} \mu_{i} \left(\sum_{g \in C_{i}} g\right)  \\ & = &  \sum_{i=1}^{s} \mu_{i} x_{i} + \sum_{i=s+1}^{r} \left(\sum_{g \in C_{i} \setminus \{x_{i}\}} \mu_{i} g\right) + \sum_{i=s+1}^{r} \mu_{i} x_{i}
\end{eqnarray*}

$\Rightarrow \sum_{i=1}^{s} \mu_{i} x_{i} + \sum_{i=s+1}^{r} \mu_{i} x_{i} + \sum_{i=s+1}^{r} \left(\sum_{g \in C_{i} \setminus \{x_{i}\}} (\mu_{i} - \lambda_{g}) g\right) = 0$. 

\noindent Then using the fact that $G$ is an $\FF $-linearly independent subset of $\FF G$ and the conjugacy classes $C_{i}$'s ($1 \leq i \leq r$) are pairwise disjoint, we get that $\mu_{i} = 0$ and $\mu_{i}-\lambda_{g}=0, \hspace{0.1cm} \forall \hspace{0.1cm} g \in C_{i} \setminus \{x_{i}\} \hspace{0.1cm} \text{and} \hspace{0.1cm} \hspace{0.1cm} \forall \hspace{0.1cm} i \in \{s+1, ..., r\}$. So $\lambda_{g} = 0, \hspace{0.1cm} \forall \hspace{0.1cm} g \in \bigcup_{i=s+1}^{r} C_{i} \setminus \{x_{i}\}$.
Therefore, the set $\mathcal{B}_{0}$ is $\FF $-linearly independent.

Suppose that $|C_{i}| = n_{i}, \hspace{0.1cm} \forall \hspace{0.1cm} i \in \{s+1, ..., r\}$. Now $$n = |G| = |\bigcup_{i=1}^{r} C_{i}| = \sum_{i=1}^{r} |C_{i}| = \sum_{i=1}^{s} |C_{i}| + \sum_{j=s+1}^{r} |C_{j}| = s + \sum_{j=s+1}^{r} n_{j}$$ so that $\sum_{j=s+1}^{r} n_{j} = n-s$.
Since the sets $C_{s+1} \setminus \{x_{s+1}\}, ..., C_{r} \setminus \{x_{r}\}$ are pairwise disjoint, therefore, \begin{eqnarray*}
|\mathcal{B}_{0}| & = & \sum_{i=s+1}^{r} |C_{i} \setminus \{x_{i}\}| = \sum_{i=s+1}^{r} (|C_{i}| - 1) = \sum_{i=s+1}^{r} (n_{i}-1) = \left(\sum_{i=s+1}^{r} n_{i}\right) - (r-s) \\ & = & (n-s) - (r-s) = n-r = \text{dim}(\der_{\inn}(\FF G)).\end{eqnarray*}
Therefore, $\mathcal{B}_{0}$ is a basis of $\der_{\inn}(\FF G)$ over $\FF $.
\end{proof}

\section{Derivations of Dihedral Group Algebras}\label{section 3}
Consider the dihedral group $$D_{2n} = \langle a, b \mid a^{n} = 1, b^{2} = 1, (ab)^{2} = 1 \rangle$$ of order $2n$, where $n \geq 3$ is a positive integer. So $D_{2n} = \{a^{i}b^{j} \mid 0 \leq i \leq n-1, 0 \leq j \leq 1\}$. In this section, we classify all derivations of the group algebra $\FF D_{2n}$ over a field $\FF $ of characteristic $0$ or an odd rational prime $p$.

\begin{lemma}[{\cite[Chapter 3]{GordonJames2003}}]\th\label{lemma 3.1}
\begin{enumerate}
\item[(i)] When $n$ is even, $D_{2n}$ has $\frac{n}{2}+3$ conjugacy classes given by $\{1\}$, $\{a^{\frac{n}{2}}\}$, $\{a^{k}, a^{-k}\}$ for $1 \leq k \leq \frac{n}{2} - 1$, $\{a^{2i}b \mid 0 \leq i \leq \frac{n}{2}-1\}$, $\{a^{2i+1}b \mid 0 \leq i \leq \frac{n}{2}-1\}$.

\item[(ii)] When $n$ is odd, $D_{2n}$ has $\frac{n+3}{2}$ conjugacy classes given by $\{1\}$, $\{a^{k}, a^{-k}\}$ for $1 \leq k \leq \frac{n-1}{2}$, $\{a^{i}b \mid 0 \leq i \leq n-1\}$.
\end{enumerate}
\end{lemma}

\begin{lemma}\th\label{lemma 3.2}
Let $\FF $ be a field of characteristic $0$ or $p$, where $p$ is an odd rational prime. Then the following statements hold.
\begin{enumerate}
\item[(i)] The set $$\mathcal{\bar{B}}(b) = \{(a^{i} - a^{-i}), ~ (a^{i} - a^{-i})b \mid i = 1, 2, ..., \lfloor \frac{n-1}{2} \rfloor\}$$ is a basis of $\bar{C}(b)$ over $\FF $.

\item[(ii)] The set $$\mathcal{\bar{B}}(ab) = \{(a^{i} - a^{-i}), ~ a(a^{i} - a^{-i})b \mid i = 1, 2, ..., \lfloor \frac{n-1}{2} \rfloor\}$$ is a basis of $\bar{C}(ab)$ over $\FF $.
\end{enumerate}
\end{lemma}
\begin{proof}
(i) Let $\alpha \in \bar{C}(b)$. Then from $\alpha b = - b \alpha$ and using the fact that $D_{2n}$ is an $\FF $-basis of $\FF D_{2n}$, we get that $\alpha$ is indeed an $\FF $-linear combination of the elements in $\mathcal{\bar{B}}(b)$. Also, the set $\mathcal{\bar{B}}(b)$ is linearly independent over $\FF $ and hence forms a basis of $\bar{C}(b)$ over $\FF $.

(ii) Define a map $\theta : D_{2n} \rightarrow D_{2n}$ by $\theta(a^{i}b^{j}) = a^{i}(ab)^{j}, \hspace{0.1cm} \forall \hspace{0.1cm} i \in \{0, 1, ..., n-1\}, j \in \{0, 1\}$. Then $\theta$ is an automorphism of $D_{2n}$. Extend $\theta$ $\FF $-linearly to an $\FF $-algebra automorphism of $\FF D_{2n}$. Let $\alpha \in \FF D_{2n}$. We can write $\alpha = \beta + \gamma b$ for some $\beta, \gamma \in \FF  \langle a \rangle$. If $\alpha \in \bar{C}(b)$, then $\alpha b = - b \alpha$ implies that $\beta b = - b \beta$ and $\gamma b = - b \gamma$. Also, $\theta(\alpha) = \beta + \gamma ab$. Therefore, $\theta(\alpha) (ab) = -(ab) \theta(\alpha)$ so that $\theta(\alpha) \in \bar{C}(ab)$. Conversely, suppose $\theta(\alpha) \in \bar{C}(ab)$. Then $\theta(\alpha) (ab) = -(ab) \theta(\alpha)$ gives $\beta b = -b\beta$ and $\gamma b = -b\gamma$ so that $\alpha b = -b \alpha$ and hence $\alpha \in \bar{C}(b)$.
Therefore, $\alpha \in \bar{C}(b)$ if and only if $\theta(\alpha) \in \bar{C}(ab)$. Since $\theta$ is an automorphism of $\FF D_{2n}$ and $\mathcal{\bar{B}}(b)$ is an $\FF $-basis of $\bar{C}(b)$, therefore, $\theta(\mathcal{\bar{B}}(b)) = \mathcal{\bar{B}}(ab)$ becomes an $\FF $-basis of $\bar{C}(ab)$.
\end{proof}

We know that $\der_{\FF }(\FF G)$ forms a vector space over $\FF $. In the theorem below, we determine the dimension and a basis of this vector space.

\begin{theorem}\th\label{theorem 3.3}
Let $\FF $ be a field and $p$ be an odd rational prime. Then the following statements hold.
\begin{enumerate}
\item[(i)] If $\ch(\FF ) = 0$ or $p$ with $\gcd(n,p)=1$, then $\der_{\FF }(\FF D_{2n})$ has dimension $3 \lfloor \frac{n-1}{2} \rfloor$ over $\FF $ and a basis 
\begin{equation*}
\begin{aligned}
\mathcal{B} & = \{d_{(\bar{a}, \bar{b})} \mid (\bar{a}, \bar{b}) \in \{((a^{i} - a^{-i})b,0), ~~ (a(a^{i} - a^{-i})b, (a^{i} - a^{-i})), ~~ (0, (a^{i} - a^{-i})b) \\ &\quad \mid i = 1, 2, ..., \lfloor \frac{n-1}{2} \rfloor\}\}.
\end{aligned}
\end{equation*} 
\item[(ii)] If $\ch(\FF ) = p$ with $\gcd(n,p) \neq 1$, then $\der_{\FF }(\FF D_{2n})$ has dimension $4 \lfloor \frac{n-1}{2} \rfloor$ over $\FF $ and a basis 
\begin{equation*}
\begin{aligned}
\mathcal{B}' & = \{d_{(\bar{a}, \bar{b})} \mid (\bar{a}, \bar{b}) \in \{(a(a^{i} - a^{-i}),0), ~~ (a(a^{i} - a^{-i}), (a^{i} - a^{-i})b), ~~ ((a^{i}-a^{-i})b, 0), \\ &\quad (a(a^{i}-a^{-i})b, (a^{i}-a^{-i})) \mid i = 1, 2, ..., \lfloor \frac{n-1}{2}\rfloor\}\}.
\end{aligned}
\end{equation*}
\end{enumerate}
\end{theorem}

\begin{proof}
The relators of the group $D_{2n}$ are $a^{n}, b^{2}$ and $(ab)^{2}$.  Let $f:X = \{a, b\} \rightarrow \FF D_{2n}$ be a map that can be extended to an $\FF $-derivation of $\FF D_{2n}$. By \th\ref{theorem 2.11}, this is possible if and only if $\tilde{f}(a^{n}) = 0, \hspace{0.1cm} \tilde{f}(b^{2}) = 0, \hspace{0.1cm} \text{and} \hspace{0.1cm} \tilde{f}((ab)^{2}) = 0$. We can write $f(a)$ as $f(a) = \alpha + \beta b$ for some $\alpha, \beta \in \FF  \langle a \rangle$. Then using (\ref{eq 2.3}),
\begin{eqnarray*}
\tilde{f}(a^{n}) & = & \sum_{i=1}^{n} \left(\prod_{j=1}^{i-1} a \right) \tilde{f}(a) \left(\prod_{j=i+1}^{n} a \right) = \sum_{j=0}^{n-1} a^{j} f(a) a^{n-(j+1)} = \sum_{j=0}^{n-1} a^{j} (\alpha + \beta b) a^{n-(j+1)}. 
\end{eqnarray*}

\begin{equation}\label{eq 3.1}
\begin{aligned}
 \Rightarrow \tilde{f}(a^{n}) = n \alpha a^{n-1} + \sum_{j=0}^{n-1}a^{2j+1} \beta b.
\end{aligned} 
 \end{equation}
Right multiplying both sides of $\tilde{f}(a^{n}) = 0$ by $a$, we get that $\tilde{f}(a^{n}) = 0$ if and only if $n \alpha = 0$ and $\sum_{j=0}^{n-1}a^{2j} \beta = 0$. Also, $\tilde{f}(b^{2}) = \tilde{f}(b) \phi(b) + \phi(b) \tilde{f}(b) = f(b) b + b f(b)$ and $\tilde{f}((ab)^{2}) = \tilde{f}(ab) \phi(ab) + \phi(ab) \tilde{f}(ab) = \tilde{f}(ab) (ab) + (ab) \tilde{f}(ab)$. So $\tilde{f}(b^{2}) = 0$ if and only if $f(b) \in \tilde{C}(b)$ and $\tilde{f}((ab)^{2}) = 0$ if and only if $\tilde{f}(ab) \in \bar{C}(ab)$. Now, there are two possibilities (i) and (ii).\vspace{10pt}

\textbf{(i)} $\ch(\FF ) = 0$ or $p$ with $\gcd(n,p)=1$: Then $n \alpha = 0$ if and only if $\alpha = 0$. Again, we have two possible cases: $n$ even and $n$ odd.

\textbf{Case 1:} $n$ is even. Then $\sum_{j=0}^{n-1}a^{2j} \beta = 0$ if and only if $\left( \sum_{j=0}^{\frac{n}{2}-1} a^{2j} \beta \right) = 0$. Let $\beta = \sum_{i=0}^{n-1} \lambda_{i} a^{i}$ for some $\lambda_{i} \in \FF $ ($0 \leq i \leq n-1$). Then,
\begin{eqnarray*}
0 & = & \sum_{j=0}^{\frac{n}{2}-1} a^{2j} \beta = \left( \sum_{\substack{i=0 \\ i \hspace{0.1cm} \text{even}}}^{n-1} \lambda_{i} \right) \left(\sum_{j=0}^{\frac{n}{2}-1} a^{2j}\right) + \left( \sum_{\substack{i=0 \\ i \hspace{0.1cm} \text{odd}}}^{n-1} \lambda_{i} \right) \left(\sum_{j=0}^{\frac{n}{2}-1} a^{2j+1}\right).
\end{eqnarray*}
So $\sum_{\substack{i=0 \\ i \hspace{0.1cm} \text{even}}}^{n-1} \lambda_{i} = 0$ and $\sum_{\substack{i=0 \\ i \hspace{0.1cm} \text{odd}}}^{n-1} \lambda_{i} = 0$ so that $\beta \in \Delta '(\langle  a^{2} \rangle)$. Therefore, in this case, $f(a) = \beta b$, where $\beta \in \Delta '(\langle  a^{2} \rangle)$.\vspace{6pt}  

\textbf{Case 2:} $n$ is odd. Then $\sum_{j=0}^{n-1}a^{2j}\beta = 0$ implies that $\sum_{j=0}^{n-1}a^{j}\beta = 0$. Therefore, \begin{eqnarray*}
0 & = & \sum_{j=0}^{n-1} a^{j} \beta = \sum_{i=0}^{n-1} \left(\lambda_{i} \left(\sum_{j=0}^{n-1} a^{i+j}\right)\right) = \left(\sum_{i=0}^{n-1} \lambda_{i}\right)\left(\sum_{k=0}^{n-1} a^{k}\right).
\end{eqnarray*}
So $\sum_{i=0}^{n-1}\lambda_{i} = 0$ so that $\beta \in \Delta(\langle a \rangle)$. Therefore, $f(a) = \beta b$, where $\beta \in \Delta(\langle a \rangle)$.

Now we come out of the two cases and proceed for a general proof. Since $f(a) = \beta b$, so $\beta = \tilde{f}(ab) - af(b)$. Since $f(b) \in \bar{C}(b)$ and by \th\ref{lemma 3.2} (i), $\mathcal{\bar{B}}(b)$ is a basis of $\bar{C}(b)$ over $\FF $, so \begin{equation}\label{eq 3.2}
f(b) = \sum_{i=1}^{\lfloor \frac{n-1}{2}\rfloor} \mu_{i} (a^{i}-a^{-i}) + \sum_{i=1}^{\lfloor \frac{n-1}{2} \rfloor} \nu_{i} (a^{i}-a^{-i})b\end{equation} for some $\mu_{i}, \nu_{i} \in \FF $ ($1 \leq i \leq \lfloor \frac{n-1}{2} \rfloor$). Further since $\tilde{f}(ab) \in \bar{C}(ab)$ and by \th\ref{lemma 3.2} (ii), $\mathcal{\bar{B}}(ab)$ is a basis of $\bar{C}(ab)$ over $\FF $, so \begin{equation}\label{eq 3.3}
\tilde{f}(ab) = \sum_{i=1}^{\lfloor \frac{n-1}{2} \rfloor} \delta_{i} (a^{i}-a^{-i}) + \sum_{i=1}^{\lfloor \frac{n-1}{2} \rfloor} \gamma_{i} a(a^{i}-a^{-i})b\end{equation} for some $\delta_{i}, \gamma_{i} \in \FF $ ($1 \leq i \leq \lfloor \frac{n-1}{2} \rfloor$). Therefore, $\beta = \tilde{f}(ab) - af(b)$ implies that $$\beta = \sum_{i=1}^{\lfloor \frac{n-1}{2} \rfloor} \delta_{i} (a^{i}-a^{-i}) + \sum_{i=1}^{\lfloor \frac{n-1}{2} \rfloor} \gamma_{i} a(a^{i}-a^{-i})b - \sum_{i=1}^{\lfloor \frac{n-1}{2} \rfloor} \mu_{i} a(a^{i}-a^{-i}) - \sum_{i=1}^{\lfloor \frac{n-1}{2} \rfloor} \nu_{i} a(a^{i}-a^{-i})b.$$
But then the fact that $\beta \in \FF  \langle a \rangle$ implies that $\gamma_{i} = \nu_{i}, \hspace{0.1cm} \forall \hspace{0.1cm} i \in \{1, ...,  \lfloor \frac{n-1}{2} \rfloor\}$. Therefore, $$\beta = \sum_{i=1}^{\lfloor \frac{n-1}{2} \rfloor} \delta_{i} (a^{i}-a^{-i}) - \sum_{i=1}^{\lfloor \frac{n-1}{2} \rfloor} \mu_{i} a(a^{i}-a^{-i}).$$ Note that the above expression for $\beta$ is indeed in $\Delta '(\langle a^{2} \rangle)$ and $\Delta(\langle a \rangle)$.
So $f(a) = \beta b$ where $\beta$ is as found above. Put \begin{equation*}
\begin{aligned}
\mathcal{B} & = \{d_{(\bar{a}, \bar{b})} \mid (\bar{a}, \bar{b}) \in \{((a^{i} - a^{-i})b,0), ~~ (a(a^{i} - a^{-i})b, (a^{i} - a^{-i})), ~~ (0, (a^{i} - a^{-i})b) \\ &\quad \mid i = 1, 2, ..., \lfloor \frac{n-1}{2} \rfloor\}\}.
\end{aligned}
\end{equation*}
We prove that $\mathcal{B}$ is a basis of $\der_{\FF }(\FF D_{2n})$ over $\FF $. If $c_{1}, c_{2} \in \FF $ and $\alpha_{1}, \alpha_{2}, \beta_{1}, \beta_{2} \in \FF D_{2n}$, then 
\begin{eqnarray*}\left(c_{1}d_{(\alpha_{1}, \beta_{1})} + c_{2}d_{(\alpha_{2}, \beta_{2})}\right)(a) & = & c_{1}d_{(\alpha_{1}, \beta_{1})}(a) + c_{2}d_{(\alpha_{2}, \beta_{2})}(a) = c_{1}\alpha_{1} + c_{2}\alpha_{2} \\ & = & d_{(c_{1}\alpha_{1} + c_{2} \alpha_{2}, c_{1}\beta_{1} + c_{2} \beta_{2})}(a).\end{eqnarray*}
Similarly, $d_{(c_{1}\alpha_{1} + c_{2} \alpha_{2}, c_{1}\beta_{1} + c_{2} \beta_{2})}(b) = \left(c_{1}d_{(\alpha_{1}, \beta_{1})} + c_{2}d_{(\alpha_{2}, \beta_{2})}\right)(b)$. So by using the fact that any $\FF $-derivation of $\FF D_{2n}$ is completely determined by its image values $d(a)$ and $d(b)$, we get from above that $c_{1}d_{(\alpha_{1}, \beta_{1})} + c_{2}d_{(\alpha_{2}, \beta_{2})} = d_{(c_{1}\alpha_{1} + c_{2} \alpha_{2}, c_{1}\beta_{1} + c_{2} \beta_{2})}$. Therefore, $\mathcal{B}$ spans $\der_{\FF }(\FF D_{2n})$ over $\FF $. Now let $r_{i}, s_{i}, t_{i} \in \FF $ ($1 \leq i \leq \lfloor \frac{n-1}{2} \rfloor$) such that $$\sum_{i=1}^{\lfloor \frac{n-1}{2} \rfloor} r_{i}d_{((a^{i}-a^{-i})b,0)} + \sum_{i=1}^{\lfloor \frac{n-1}{2} \rfloor} s_{i}d_{(a(a^{i}-a^{-i})b,(a^{i}-a^{-i}))} +\sum_{i=1}^{\lfloor \frac{n-1}{2} \rfloor} t_{i}d_{(0,(a^{i}-a^{-i}))} = 0,$$ where $0$ on the right denotes the zero derivation.

$\Rightarrow d_{\left(\sum_{i=1}^{\lfloor \frac{n-1}{2} \rfloor} r_{i}(a^{i}-a^{-i})b + \sum_{i=1}^{\lfloor \frac{n-1}{2} \rfloor} s_{i}a(a^{i}-a^{-i})b, \sum_{i=1}^{\lfloor \frac{n-1}{2} \rfloor} s_{i}(a^{i}-a^{-i}) + \sum_{i=1}^{\lfloor \frac{n-1}{2} \rfloor} t_{i}(a^{i}-a^{-i})\right)} = 0$. 

\noindent This implies that values of the function on the left take zero value at all points in the domain $\FF D_{2n}$, in particular, at $a$ and $b$. Therefore,

$\sum_{i=1}^{\lfloor \frac{n-1}{2} \rfloor} r_{i}(a^{i}-a^{-i})b + \sum_{i=1}^{\lfloor \frac{n-1}{2} \rfloor} s_{i}a(a^{i}-a^{-i})b = 0$ and $\sum_{i=1}^{\lfloor \frac{n-1}{2} \rfloor} s_{i}(a^{i}-a^{-i}) + \sum_{i=1}^{\lfloor \frac{n-1}{2} \rfloor} t_{i}(a^{i}-a^{-i}) = 0$.

But the $\FF $-linear independence of the elements $(a^{i}-a^{-i})$, $(a^{i}-a^{-i})b$, $a(a^{i}-a^{-i})b$ ($1 \leq i \leq \lfloor \frac{n-1}{2} \rfloor$) implies that $r_{i} = 0$, $s_{i} = 0$ and $t_{i} = 0$ for all $i \in \{1, 2, ..., \lfloor \frac{n-1}{2} \rfloor\}$. Therefore, $\mathcal{B}$ is an $\FF $-linearly independent subset of $\der_{\FF }(\FF D_{2n})$. Hence, $\mathcal{B}$ is an $\FF $-basis of $\der_{\FF }(\FF D_{2n})$. Further since $|\mathcal{B}| = 3\lfloor \frac{n-1}{2}\rfloor$, therefore, the dimension of $\der_{\FF }(\FF D_{2n})$ over $\FF $ is $3 \lfloor \frac{n-1}{2}\rfloor$.\vspace{10pt}

\textbf{(ii)} $\ch(\FF ) = p$ with $\gcd(n,p)=p$: Then $n \alpha a^{n-1} = 0$. Therefore, in this case, (\ref{eq 3.1}) becomes $\tilde{f}(a^{n}) = \sum_{j=0}^{n-1}a^{2j+1}\beta b$. Now proceeding as in (i), it can be shown that when $n$ is even, $\tilde{f}(a^{n}) = 0$ if and only if $\beta \in \Delta '(\langle a^{2} \rangle)$ and when $n$ is odd, $\tilde{f}(a^{n}) = 0$ if and only if $\beta \in \Delta(\langle a \rangle)$. Therefore, $f(a) = \alpha + \beta b$, where $\begin{cases}
\beta \in \Delta '(\langle a^{2} \rangle) & \text{if $n$ is even} \\
\beta \in \Delta(\langle a \rangle) & \text{if $n$ is odd}
\end{cases}$. Now using $f(a)b = \tilde{f}(ab) - af(b)$, (\ref{eq 3.2}) and (\ref{eq 3.3}), we get that $$\alpha = \sum_{i=1}^{\lfloor \frac{n-1}{2} \rfloor} \gamma_{i} a(a^{i}-a^{-i}) - \sum_{i=1}^{\lfloor \frac{n-1}{2} \rfloor} \nu_{i} a(a^{i}-a^{-i}) \hspace{0.1cm} \text{and} \hspace{0.1cm} \beta = \sum_{i=1}^{\lfloor \frac{n-1}{2} \rfloor} \delta_{i} (a^{i}-a^{-i}) - \sum_{i=1}^{\lfloor \frac{n-1}{2} \rfloor} \mu_{i} a(a^{i}-a^{-i}).$$
Note that the above expression for $\beta$ indeed belongs to $\Delta ' (\langle a^{2} \rangle)$ when $n$ is even and to $\Delta(\langle a \rangle)$ when $n$ is odd. Therefore, $f(a) = \alpha + \beta b$ where $\alpha$ and $\beta$ are as found above. Put \begin{equation*}
\begin{split}
\mathcal{B}' & = \{d_{(\bar{a}, \bar{b})} \mid (\bar{a}, \bar{b}) \in \{(a(a^{i} - a^{-i}),0), ~~ (a(a^{i} - a^{-i}), (a^{i} - a^{-i})b), ~~ ((a^{i}-a^{-i})b, 0), \\ &\quad (a(a^{i}-a^{-i})b, (a^{i}-a^{-i})) \mid i = 1, 2, ..., \lfloor \frac{n-1}{2}\rfloor\}\}.
\end{split}
\end{equation*}
Then as shown in part (i), it can be shown that $\mathcal{B}'$ is an $\FF $-basis of $\der_{\FF }(\FF D_{2n})$. Also, since $|\mathcal{B}'| = 4\lfloor \frac{n-1}{2}\rfloor$, therefore, the dimension of $\der_{\FF }(\FF D_{2n})$ over $\FF $ is $4 \lfloor \frac{n-1}{2}\rfloor$. Hence proved.
\end{proof}

\begin{corollary}\th\label{corollary 3.4}
Let $\FF $ be an algebraic extension of a prime field such that $\ch(\FF )$ is either $0$ or an odd rational prime $p$ relatively prime to $n$. Then $\der(\FF D_{2n})$ has dimension $3 \lfloor \frac{n-1}{2} \rfloor$ over $\FF $ and a basis $\mathcal{B}$ given in part (i) of \th\ref{theorem 3.3}.
\end{corollary}
\begin{proof}
By \th\ref{theorem 2.12}, every $\FF $-derivation of $\FF D_{2n}$ is a derivation of $\FF D_{2n}$. Now the proof immediately follows from \th\ref{theorem 3.3} (i).
\end{proof}

\begin{theorem}\th\label{theorem 3.5}
Let $\FF $ be a field. Then $\der_{\inn}(\FF D_{2n})$ has dimension $3 \lfloor \frac{n-1}{2} \rfloor$ over $\FF $.

\begin{enumerate}
\item[(i)] When $n$ is even, an $\FF $-basis for 
$\der_{\inn}(\FF D_{2n})$ is 
\begin{equation*}
\begin{aligned}
\mathcal{B}_{e} & = \{d_{g} \mid g \in \{a^{i}, ~ a^{2i}b, ~ a^{2i+1}b \mid 1 \leq i \leq \frac{n}{2}-1\}\}.
\end{aligned}
\end{equation*}

\item[(ii)] When $n$ is odd, an $\FF $-basis for $\der_{\inn}(\FF D_{2n})$ is 
\begin{equation*}
\begin{aligned}
\mathcal{B}_{o} & = \{d_{g} \mid g \in \{a^{i} \mid 1 \leq i \leq \frac{n-1}{2}\} \cup \{a^{i}b \mid 1 \leq i \leq n-1\}\}.
\end{aligned}
\end{equation*}
\end{enumerate}
\end{theorem}
\begin{proof}
Follows from \th\ref{theorem 2.15} and \th\ref{lemma 3.1}.
\end{proof}

\begin{corollary}\th\label{corollary 3.6}
Let $\FF $ be an algebraic extension of a prime field and $p$ be an odd rational prime.
\begin{enumerate}
\item[(i)] If $\ch(\FF )$ is either $0$ or $p$ with $\gcd(n,p)=1$, then all derivations of $\FF D_{2n}$ are inner, that is,  $\der(\FF D_{2n}) = \der_{\inn}(\FF D_{2n})$. In other words, $\FF D_{2n}$ has no non-zero outer derivations.

\item[(ii)] If $\ch(\FF ) = p$ with $\gcd(n,p) \neq 1$, then $\der_{\inn}(\FF D_{2n}) \subsetneq \der(\FF D_{2n})$. In other words, $\FF D_{2n}$ has non-zero outer derivations.
\end{enumerate}
\end{corollary}
\begin{proof}
Follows from \th\ref{theorem 3.3} (ii), \th\ref{corollary 3.4} and \th\ref{theorem 3.5}.
\end{proof}

\section{Derivations of Dicyclic Group Algebras}\label{section 4}
The group $T_{4n}$ has the presentation given by $$T_{4n} = \langle a, b \mid a^{2n}=1, a^{n}=b^{2}, b^{-1}ab = a^{-1} \rangle.$$ So $T_{4n} = \{a^{i}b^{j} \mid 0 \leq i \leq 2n-1, 0 \leq j \leq 1\}$. In this section, we classify all derivations of the group algebra $\FF T_{4n}$ over a field $\FF $ of characteristic $0$ or an odd rational prime $p$. $T_{4n}$ is the well-known generalized quaternion group when $n$ is a power of $2$.

\begin{lemma}[{\cite[Proposition 2.2]{Salahshour2020}}]\th\label{lemma 4.1}
$T_{4n}$ has precisely $n+3$ conjugacy classes given by $$\{1\}, ~~ \{a^{n}\}, ~~ \{a^{i}, a^{-i}\} ~ \text{for} ~ 1 \leq i \leq n-1, ~~ \{a^{2i}b \mid 0 \leq i \leq n-1\}, ~~ \{a^{2i+1}b \mid 0 \leq i \leq n-1\}.$$
\end{lemma}

\begin{lemma}\th\label{lemma 4.2}
Let $\FF $ be a field of characteristic $0$ or $p$, where $p$ is an odd rational prime. Then the following statements hold.
\begin{enumerate} 
\item[(i)] The set $\mathcal{\bar{B}}(b) = \{(a^{i}-a^{-i}), ~ (a^{i}-a^{-i})b \mid i \in \{1, 2, ..., n-1\}\}$ is a basis of $\bar{C}(b)$ over $\FF $.

\item[(ii)] The set $\mathcal{\bar{B}}(a^{n+1}b) = \{(a^{i}-a^{-i}), ~ a^{n+1}(a^{i}-a^{-i})b \mid i \in \{1, 2, ..., n-1\}\}$ is a basis of $\bar{C}(a^{n+1}b)$ over $\FF $.
\end{enumerate}
\end{lemma}
\begin{proof}
(i) $\mathcal{\bar{B}}(b)$ is clearly an $\FF $-linearly independent subset of $\bar{C}(b)$. Now let $\alpha \in \bar{C}(b)$ and $\alpha = \sum_{i,j}\lambda_{i,j} a^{i}b^{j}$ for some $\lambda_{i,j} \in \FF $ ($0 \leq i \leq 2n-1$, $0 \leq j \leq 1$). Also, put $\lambda_{(2n),j} = \lambda_{0,j}$ for $j \in \{0,1\}$. From $\alpha b = - b \alpha$, we get that $\lambda_{i,j} = - \lambda_{k,j}$ for all $i, k \in \{0, 1, ..., 2n-1\}$ with $i+k = 2n$ and for all $j \in \{0,1\}$. In particular, $\lambda_{0,j} = 0$ and $\lambda_{n,j} = 0$ for $j \in \{0,1\}$. So $\alpha = \sum_{i=1}^{n-1}\lambda_{i,0}(a^{i}-a^{-i}) + \sum_{i=1}^{n-1}\lambda_{i,1}(a^{i}-a^{-i})b$. Therefore, $\mathcal{\bar{B}}(b)$ spans $\bar{C}(b)$ over $\FF $ and hence is an $\FF $-basis of $\bar{C}(b)$.

(ii) The map $\theta : T_{4n} \rightarrow T_{4n}$ defined by $\theta(a^{i}b^{j}) = a^{i}(a^{n+1}b)^{j}, \hspace{0.1cm} \forall \hspace{0.1cm} i \in \{0, 1, ..., 2n-1\}, j \in \{0, 1\}$ is an $\FF $-algebra automorphism of $\FF T_{4n}$. Now the proof follows easily on similar lines as the proof of \th\ref{lemma 3.2}.
\end{proof}

\begin{theorem}\th\label{theorem 4.3}
Let $\FF $ be a field and $p$ be an odd rational prime. Then the following statements hold.
\begin{enumerate}
\item[(i)] If $\ch(\FF ) = 0$ or $p$ with $\gcd(n,p)= 1$, then $\der_{\FF }(\FF T_{4n})$ has dimension $3(n-1)$ over $\FF $ and a basis 
\begin{equation*}
\begin{aligned}
\mathcal{B} & = \{d_{(\bar{a}, \bar{b})} \mid (\bar{a}, \bar{b}) \in \{(a^{n}(a^{i} - a^{-i})b,0), ~~ (a^{n+1}(a^{i} - a^{-i})b, (a^{i} - a^{-i})), ~~ (0, (a^{i} - a^{-i})b) \\ &\quad \mid i = 1, 2, ..., n-1\}\}.
\end{aligned}
\end{equation*}

\item[(ii)] If $\ch(\FF ) =p$ with $\gcd(n,p)\neq 1$, then $\der_{\FF }(\FF T_{4n})$ has dimension $4(n-1)$ over $\FF $ and a basis 
\begin{equation*}
\begin{aligned}
\mathcal{B}' & = \{d_{(\bar{a}, \bar{b})} \mid (\bar{a}, \bar{b}) \in \{(a^{n+1}(a^{i} - a^{-i}),0), ~~ (a(a^{i} - a^{-i}), (a^{i} - a^{-i})b), \\ &\quad (a^{n}(a^{i}-a^{-i})b, 0), ~~ (a^{n+1}(a^{i}-a^{-i})b, (a^{i}-a^{-i})) \mid i = 1, 2, ..., n-1\}\}.
\end{aligned}
\end{equation*}
\end{enumerate}
\end{theorem}
\begin{proof}
Let $f:X = \{a, b\} \rightarrow \FF T_{4n}$ can be extended to an $\FF $-derivation of $\FF T_{4n}$. Since the relators of the group $T_{4n}$ are $a^{2n}, a^{n}b^{2}$ and $ab^{-1}ab$, so by \th\ref{theorem 2.11}, this is possible if and only if $\tilde{f}(a^{2n}) = 0, \hspace{0.1cm} \tilde{f}(a^{n}b^{2}) = 0, \hspace{0.1cm} \text{and} \hspace{0.1cm} \tilde{f}(ab^{-1}ab) = 0$. Using (\ref{eq 3.1}), $\tilde{f}(a^{2n}) = 2a^{n}\tilde{f}(a^{n})$ so that $\tilde{f}(a^{2n}) = 0$ if and only if $\tilde{f}(a^{n})=0$. Let $f(a) = \alpha + \beta b$ for some $\alpha, \beta \in \FF  \langle a \rangle$. Then
\begin{equation}\label{eq 4.1}\tilde{f}(a^{n}) = n \alpha a^{n-1} + \sum_{j=0}^{n-1}a^{2j+1} \beta b. \end{equation}
Therefore, $\tilde{f}(a^{n}) = 0$ if and only if $n \alpha = 0$ and $\sum_{j=0}^{n-1}a^{2j} \beta = 0$. Also, $\tilde{f}(b^{2}) = f(b) b + b f(b)$ so that $\tilde{f}(b^{2}) = 0$ if and only if $f(b) \in \tilde{C}(b)$.
Using $\tilde{f}(a^{n}) = 0$, $f(a)b + af(b) = \tilde{f}(ab)$ and (\ref{eq 3.1}), we get that \begin{eqnarray*}
\tilde{f}(ab^{-1}ab) = \tilde{f}(ab)a^{n+1}b + a^{n+1}b\tilde{f}(ab).
\end{eqnarray*}
Therefore, $\tilde{f}(ab^{-1}ab) = 0$ if and only if $\tilde{f}(ab) \in \bar{C}(a^{n+1}b)$. Again, there are two possibilities (i) and (ii).\vspace{10pt}

\textbf{(i)} $\ch(\FF ) = 0$ or $p$ with $\gcd(n,p)=1$: Then $n \alpha = 0$ if and only if $\alpha = 0$. Since $\beta \in \FF \langle a \rangle$, so $\beta = \sum_{i=0}^{2n-1} \lambda_{i} a^{i}$ for some $\lambda_{i} \in \FF $ ($0 \leq i \leq 2n-1$). Also, put $\lambda_{2n} = \lambda_{0}$. Then, \begin{eqnarray*}
0 & = & \sum_{j=0}^{n-1} a^{2j} \beta = \left( \sum_{\substack{i=0 \\ i \hspace{0.1cm} \text{even}}}^{2n-1} \lambda_{i} \right) \left(\sum_{j=0}^{n-1} a^{2j}\right) + \left( \sum_{\substack{i=0 \\ i \hspace{0.1cm} \text{odd}}}^{2n-1} \lambda_{i} \right) \left(\sum_{j=0}^{n-1} a^{2j+1}\right).
\end{eqnarray*}
So $\sum_{\substack{i=0 \\ i \hspace{0.1cm} \text{even}}}^{2n-1} \lambda_{i} = 0$ and $\sum_{\substack{i=0 \\ i \hspace{0.1cm} \text{odd}}}^{2n-1} \lambda_{i} = 0$, that is, $\beta \in \Delta '(\langle a^{2} \rangle)$. Therefore, $f(a) = \beta b$, where $\beta \in \Delta '(\langle a^{2} \rangle)$. Then from $\tilde{f}(ab) = f(a)b + af(b)$, $\beta = (\tilde{f}(ab) - af(b))a^{n}$. Since $f(b) \in \bar{C}(b)$ and by \th\ref{lemma 4.2} (i), $\mathcal{\bar{B}}(b)$ is a basis of $\bar{C}(b)$ over $\FF $, so \begin{equation}\label{eq 4.2}
f(b) = \sum_{i=1}^{n-1} \mu_{i}(a^{i}-a^{-i}) + \sum_{i=1}^{n-1} \nu_{i}(a^{i}-a^{-i})b
\end{equation} for some $\mu_{i}, \nu_{i} \in \FF $ ($1 \leq i \leq n-1$). Further since $\tilde{f}(ab) \in \bar{C}(a^{n+1}b)$ and by \th\ref{lemma 4.2} (ii), $\mathcal{\bar{B}}(a^{n+1}b)$ is a basis of $\bar{C}(a^{n+1}b)$ over $\FF $, so \begin{equation}\label{eq 4.3}
\tilde{f}(ab) = \sum_{i=1}^{n-1} \delta_{i}(a^{i}-a^{-i}) + \sum_{i=1}^{n-1} \gamma_{i}a^{n+1}(a^{i}-a^{-i})b
\end{equation} for some $\delta_{i}, \gamma_{i} \in \FF $ ($1 \leq i \leq n-1$). Therefore, \begin{equation*}
\begin{split}
\beta & = \sum_{i=1}^{n-1} \delta_{i}a^{n}(a^{i}-a^{-i}) + \sum_{i=1}^{n-1} \gamma_{i}(a^{i+1}-a^{-i+1})b - \sum_{i=1}^{n-1} \mu_{i}a^{n+1}(a^{i}-a^{-i}) \\ &\quad + \sum_{i=1}^{n-1} \nu_{n-i}(a^{i+1}-a^{-i+1})b.
\end{split}
\end{equation*}
But since $\beta \in \FF \langle a \rangle$, so $\gamma_{i} = -\nu_{n-i}, \hspace{0.1cm} \forall \hspace{0.1cm} i \in \{1, 2, ..., n-1\}$. Therefore, \begin{equation*}
\begin{aligned}\beta = \sum_{i=1}^{n-1} \delta_{i}a^{n}(a^{i}-a^{-i}) - \sum_{i=1}^{n-1} \mu_{i}a^{n+1}(a^{i}-a^{-i}).
\end{aligned}
\end{equation*} So $f(a) = \beta b$ where $\beta$ is as found above. Then the set \begin{equation*}
\begin{split}
\mathcal{B} & = \{d_{(\bar{a}, \bar{b})} \mid (\bar{a}, \bar{b}) \in \{(a^{n}(a^{i} - a^{-i})b,0), ~~ (a^{n+1}(a^{i} - a^{-i})b, (a^{i} - a^{-i})), ~~ (0, (a^{i} - a^{-i})b) \\ &\quad \mid i = 1, 2, ..., n-1\}\}.
\end{split}
\end{equation*} is a basis of $\der_{\FF }(\FF T_{4n})$ over $\FF $. Since $|\mathcal{B}| = 3(n-1)$, therefore, the dimension of $\der_{\FF }(\FF T_{4n})$ over $\FF $ is $3(n-1)$.\vspace{10pt}

\textbf{(ii)} $\ch(\FF ) = p$ with $\gcd(n,p)=p$: Then (\ref{eq 4.1}) becomes $\tilde{f}(a^{n}) = \sum_{j=0}^{n-1}a^{2j+1}\beta b$. So as in (i), $\tilde{f}(a^{n}) = 0$ if and only if $\beta \in \Delta '(\langle a^{2} \rangle)$. Therefore, $f(a) = \alpha + \beta b$, where $\beta \in \Delta '(\langle a^{2} \rangle)$. Using $f(a)b = \alpha b + \beta a^{n} = \tilde{f}(ab) - af(b)$, (\ref{eq 4.2}) and (\ref{eq 4.3}), we get that $$\alpha = \sum_{i=1}^{n-1} \gamma_{i}a^{n+1}(a^{i}-a^{-i}) - \sum_{i=1}^{n-1} \nu_{i}a(a^{i}-a^{-i}) \hspace{0.1cm} \text{and} \hspace{0.1cm} \beta = \sum_{i=1}^{n-1} \delta_{i}a^{n}(a^{i}-a^{-i}) - \sum_{i=1}^{n-1} \mu_{i}a^{n+1}(a^{i}-a^{-i}).$$ Now the set \begin{equation*}
\begin{split}
\mathcal{B}' & = \{d_{(\bar{a}, \bar{b})} \mid (\bar{a}, \bar{b}) \in \{(a^{n+1}(a^{i} - a^{-i}),0), ~~ (a(a^{i} - a^{-i}), (a^{i} - a^{-i})b), ~~ (a^{n}(a^{i}-a^{-i})b, 0), \\ &\quad  (a^{n+1}(a^{i}-a^{-i})b, (a^{i}-a^{-i})) \mid i = 1, 2, ..., n-1\}\}
\end{split}
\end{equation*} becomes an $\FF $-basis of $\der(\FF T_{4n})$ so that the dimension of $\der_{\FF }(\FF T_{4n})$ over $\FF $ is $4 (n-1)$. Hence proved.
\end{proof}

\begin{corollary}\th\label{corollary 4.4}
Let $\FF $ be an algebraic extension of a prime field such that $\ch(\FF )$ is either $0$ or an odd rational prime $p$ relatively prime to $n$. Then $\der(\FF T_{4n})$ has dimension $3(n-1)$ over $\FF $ and a basis $\mathcal{B}$ as given in part (i) of \th\ref{theorem 4.3}.
\end{corollary}
\begin{proof}
Follows from \th\ref{theorem 2.12} and \th\ref{theorem 4.3} (i).
\end{proof}

\begin{theorem}\th\label{theorem 4.5}
Let $\FF $ be a field. Then $\der_{\inn}(\FF T_{4n})$ has dimension $3 (n-1)$ over $\FF $ and a basis is $\mathcal{B}_{0} = \{d_{g} \mid g \in \{a^{i}, ~ a^{2i}b, ~ a^{2i+1}b \mid 1 \leq i \leq n-1\}\}$.
\end{theorem}
\begin{proof}
Follows from \th\ref{theorem 2.15} and \th\ref{lemma 4.1}.
\end{proof}

\begin{corollary}\th\label{corollary 4.6}
Let $\FF $ be an algebraic extension of a prime field and $p$ be an odd rational prime.
\begin{enumerate}
\item[(i)] If $\ch(\FF )$ is either $0$ or $p$ with $\gcd(n,p)=1$, then all derivations of $\FF T_{4n}$ are inner, that is,  $\der(\FF T_{4n}) = \der_{\inn}(\FF T_{4n})$. In other words, $\FF T_{4n}$ has no non-zero outer derivations.

\item[(ii)] If $\ch(\FF ) = p$ with $\gcd(n,p) \neq 1$, then $\der_{\inn}(\FF T_{4n}) \subsetneq \der(\FF T_{4n})$. In other words, $\FF T_{4n}$ has non-zero outer derivations.
\end{enumerate}
\end{corollary}
\begin{proof}
Follows from \th\ref{theorem 4.3} (ii), \th\ref{corollary 4.4} and \th\ref{theorem 4.5}.
\end{proof}

\section{Derivations of Semi-Dihedral Group Algebras}\label{section 5}
The group $SD_{8n}$ has the presentation $$\langle a, b \mid a^{4n} = b^{2} = 1, bab = a^{2n-1}\rangle.$$ So $SD_{8n} = \{a^{i}b^{j} \mid 0 \leq i \leq 4n-1, 0 \leq j \leq 1\}$. In this section, we classify all derivations of the group algebra $\FF (SD_{8n})$, where $\FF $ is a field of characteristic $0$ or an odd rational prime $p$.

\begin{lemma}[{\cite[Proposition 2.5]{Salahshour2020}}]\th\label{lemma 5.1}\begin{enumerate}
\item[(i)] When $n$ is even, $SD_{8n}$ has $2n+3$ conjugacy \\ classes given by $$\{1\}, ~~ \{a^{2n}\}, ~~ \{a^{2k}, a^{-2k}\} ~ \text{for} ~ 1 \leq k \leq n-1,$$ $$\{a^{2k+1}, a^{2n-(2k+1)}\} ~ \text{for} ~ \frac{-n}{2} \leq k \leq \frac{n}{2} - 1,$$ $$b^{SD_{8n}} = \{a^{2k}b \mid 1 \leq k \leq 2n\}, ~~ (ab)^{SD_{8n}} = \{a^{2k-1}b \mid 1 \leq k \leq 2n\}.$$

\item[(ii)] When $n$ is odd, $SD_{8n}$ has $2n+6$ conjugacy classes given by $$\{1\}, ~~ \{a^{n}\}, ~~ \{a^{2n}\}, ~~ \{a^{3n}\}, ~~ \{a^{2k}, a^{-2k}\} ~ \text{for} ~ 1 \leq k \leq n-1,$$ $$\{a^{2k+1}, a^{2n-(2k+1)}\} ~ \text{for} ~ - \frac{n-1}{2} \leq k \leq \frac{n-1}{2} - 1,$$ $$b^{SD_{8n}} = \{a^{4k}b \mid 1 \leq k \leq n\}, ~~ (ab)^{SD_{8n}} = \{a^{4k+1}b \mid 1 \leq k \leq n\},$$ $$(a^{2}b)^{SD_{8n}} = \{a^{4k+2}b \mid 1 \leq k \leq n\}, ~~ (a^{3}b)^{SD_{8n}} = \{a^{4k+3}b \mid 1 \leq k \leq n\}.$$
\end{enumerate}
\end{lemma}

\begin{lemma}\th\label{lemma 5.2} 
Let $\FF $ be a field of characteristic $0$ or $p$, where $p$ is an odd rational prime. Then the following statements hold.
\begin{enumerate}
\item[(i)] The set 
\begin{equation*}
\begin{aligned}
\mathcal{\bar{B}}(b) & = \{(a^{2k} - a^{-2k}), ~ (a^{2k} - a^{-2k})b \mid 1 \leq k \leq n-1\} \\ &\quad \cup \{(a^{2k+1} - a^{2n-(2k+1)}), ~ (a^{2k+1} - a^{2n-(2k+1)})b \mid - \lfloor \frac{n}{2} \rfloor \leq k \leq \lfloor \frac{n}{2} \rfloor - 1\}
\end{aligned}
\end{equation*}
is a basis of $\bar{C}(b)$ over $\FF $.

\item[(ii)] The set 
\begin{equation*}
\begin{aligned}
\mathcal{\bar{B}}(ab) & = \{(a^{2k} - a^{-2k}), ~ a(a^{2k} - a^{-2k})b \mid 1 \leq k \leq n-1\} \\ &\quad \cup \{(a^{2k+1} - a^{2n-(2k+1)}), ~ a(a^{2k+1} - a^{2n-(2k+1)})b \mid - \lfloor \frac{n}{2} \rfloor \leq k \leq \lfloor \frac{n}{2} \rfloor - 1\}
\end{aligned}
\end{equation*}
is a basis of $\bar{C}(ab)$ over $\FF $.
\end{enumerate}
\end{lemma}

\begin{proof}
(i) $\mathcal{\bar{B}}(b)$ is an $\FF $-linearly independent subset of $\bar{C}(b)$. Now let $\alpha \in \bar{C}(b)$ and $\alpha = \sum_{\substack{1 \leq i \leq 4n \\ 0 \leq j \leq 1}} \lambda_{i,j} a^{i}b^{j}$ for some $\lambda_{i,j} \in \FF $ ($1 \leq i \leq 4n, 0 \leq j \leq 1$). Put $\lambda_{(2n),j} = \lambda_{0,j}$ for $j \in \{0, 1\}$. Since $\alpha b = \sum_{0 \leq i \leq 4n-1} \lambda_{i,0} a^{i}b + \sum_{1 \leq i \leq 4n} \lambda_{i,1} a^{i}$ and $b \alpha = (\sum_{\substack{0 \leq k \leq 4n-1 \\ \text{$k$ even}}} \lambda_{(4n-k),0} a^{k}b + \sum_{\substack{0 \leq k \leq 4n-1 \\ \text{$k$ odd}}} \lambda_{(2n+(4n-k)),0} a^{k}b) + (\sum_{\substack{0 \leq k \leq 4n-1 \\ \text{$k$ even}}} \lambda_{(4n-k),1} a^{k} + \sum_{\substack{0 \leq k \leq 4n-1 \\ \text{$k$ odd}}} \lambda_{(2n+(4n-k)),1} a^{k})$, so $\alpha b = -b \alpha$ implies that for all $i \in \{0, 1, ..., 4n-1\}$, $j \in \{0,1\}$, $\lambda_{i,j} = \begin{cases} 
- \lambda_{(4n-i),j} & \text{if $i$ is even} \\
- \lambda_{(2n+(4n-i)),j} & \text{if $i$ is odd}
\end{cases}.$ Therefore, $\alpha = \sum_{\substack{1 \leq k \leq n-1 \\ 0 \leq j \leq 1}} \lambda_{k,j}(a^{2k} - a^{-2k})b^{j} + \sum_{\substack{- \lfloor \frac{n}{2} \rfloor \leq k \leq \lfloor \frac{n}{2} \rfloor - 1 \\ 0 \leq j \leq 1}} \lambda_{k,j}(a^{2k+1} - a^{2n-(2k+1)})b^{j}$. Hence, $\mathcal{\bar{B}}(b)$ is an $\FF $-basis of $\bar{C}(b)$.

(ii) The map $\theta:SD_{8n} \rightarrow SD_{8n}$ defined by $\theta(a^{i}b^{j}) = a^{i}(ab)^{j}, \hspace{0.1cm} \forall \hspace{0.1cm} 0 \leq i \leq 4n-1 \hspace{0.1cm} \text{and} \hspace{0.1cm} 0 \leq j \leq 1$ is an $\FF $-algebra automorphism of $SD_{8n}$. Now the proof follows easily on the similar lines of the proof of \th\ref{lemma 3.2}.
\end{proof}

\begin{theorem}\th\label{theorem 5.3}
Let $\FF $ be a field and $p$ be an odd rational prime. Then the following statements hold.
\begin{enumerate}
\item[(i)] If $\ch(\FF ) = 0$ or $p$ with $\gcd(n,p)=1$, then the dimension of $\der_{\FF }(\FF (SD_{8n}))$ over $\FF $ is $3(2n-1)$ if $n$ is even and $6(n-1)$ if $n$ is odd, and a basis 
\begin{equation*}
\begin{aligned}
\mathcal{B} & = \{d_{(\bar{a}, \bar{b})} \mid (\bar{a}, \bar{b}) \in \{(a^{2s} - a^{-2s})b,0), ~~ ((a^{2t+1} - a^{2n-(2t+1)})b,0), \\ &\quad (a(a^{2s} - a^{-2s})b, (a^{2s}-a^{-2s})), ~~ (a(a^{2t+1} - a^{2n-(2t+1)})b, (a^{2t+1} - a^{2n-(2t+1)})b), \\ &\quad (0, (a^{2s} - a^{-2s})b), ~~ (0, (a^{2t+1} - a^{2n-(2t+1)})b) \mid 1 \leq s \leq n-1, \\ &\quad -\lfloor \frac{n}{2} \rfloor \leq t \leq \lfloor \frac{n}{2} \rfloor -1\}\}.
\end{aligned}
\end{equation*}

\item[(ii)] If $\ch(\FF ) = p$ with $\gcd(n,p) \neq 1$, then the dimension of $\der_{\FF }(\FF (SD_{8n}))$ over $\FF $ is $4(2n-1)$ if $n$ is even and $8(n-1)$ if $n$ is odd, and a basis 
\begin{equation*}
\begin{aligned}
\mathcal{B}' & = \{d_{(\bar{a}, \bar{b})} \mid (\bar{a}, \bar{b}) \in \{(a(a^{2s} - a^{-2s}),0), ~~ (a(a^{2t+1} - a^{2n-(2t+1)}),0), \\ &\quad (a(a^{2s} - a^{-2s}), (a^{2s} - a^{-2s})b), ~~ (a(a^{2t+1} - a^{2n-(2t+1)}), (a^{2t+1} - a^{2n-(2t+1)})b), \\ &\quad ((a^{2s} - a^{-2s})b,0), ~~ ((a^{2t+1} - a^{2n-(2t+1)})b,0), ~~ (a(a^{2s} - a^{-2s})b, (a^{2s} - a^{-2s})), \\ &\quad (a(a^{2t+1} - a^{2n-(2t+1)})b, (a^{2t+1} - a^{2n-(2t+1)})) \mid 1 \leq s \leq n-1, \\ &\quad -\lfloor \frac{n}{2} \rfloor \leq t \leq \lfloor \frac{n}{2} \rfloor -1\}\}.
\end{aligned}
\end{equation*}
\end{enumerate} 
\end{theorem}

\begin{proof}
Let $f:X = \{a, b\} \rightarrow \FF (SD_{8n})$ be a map that can be extended to an $\FF $-derivation of $\FF (SD_{8n})$. By \th\ref{theorem 2.11}, this holds if and only if $\tilde{f}(a^{4n}) = 0, \hspace{0.2cm} \tilde{f}(b^{2}) = 0, \hspace{0.1cm} \text{and} \hspace{0.1cm} \tilde{f}(a^{2n+1}bab) \\ = 0$. Since $\tilde{f}(a^{4n}) = 2a^{2n} \tilde{f}(a^{2n})$, so $\tilde{f}(a^{4n}) = 0 \Leftrightarrow \tilde{f}(a^{2n}) = 0$. Write $f(a)$ as $f(a) = \alpha + \beta b$ for some $\alpha, \beta \in \FF  \langle a \rangle$. So \begin{equation}\label{eq 5.3} \tilde{f}(a^{2n}) = 2n \alpha a^{2n-1} + \beta \left(\sum_{\substack{1 \leq i \leq 2n \\ \text{$i$ even}}} a^{2n+2i-1}b + \sum_{\substack{1 \leq i \leq 2n \\ \text{$i$ odd}}} a^{2i-1}b \right). \end{equation}
Let $\beta = \sum_{i=1}^{4n} \lambda_{i} a^{i}$ for some $\lambda_{i} \in \FF $ ($1 \leq i \leq 4n$) with $\lambda_{0} = \lambda_{4n}$. Also, $\tilde{f}(b^{2}) = f(b)b + bf(b)$ and \begin{equation}\label{eq 5.4}
\tilde{f}(a^{2n+1}bab) = \tilde{f}(a^{2n})abab + a^{2n}f(a)bab + a^{2n+1}f(b)ab + a^{2n+1}bf(a)b + a^{2n+1}baf(b).
\end{equation}

Again, there are two possibilities (i) and (ii).\vspace{10pt}

\textbf{(i)} $\ch(\FF ) = 0$ or $p$ with $\gcd(n,p)=1$: There are two possible cases: $n$ even and $n$ odd.

\textbf{Case 1:} $n$ is even. Then $n = 2m$ for some $m \in \mathbb{N}$. So from (\ref{eq 5.3}), we get that \begin{eqnarray*}\tilde{f}(a^{2n}) & = & 2n \alpha a^{2n-1} + \beta \left(\sum_{k=1}^{n} a^{4(m+k)-1}b + \sum_{k=1}^{n} a^{4k-3}b \right) = 2n \alpha a^{2n-1} + \beta \left(\sum_{\substack{1 \leq i \leq 4n \\ \text{$i$ odd}}} a^{i}b \right) .
\end{eqnarray*}
So $\tilde{f}(a^{2n}) = 0$ if and only if $\alpha = 0$ and $\beta \left(\sum_{\substack{1 \leq i \leq 4n \\ \text{$i$ odd}}} a^{i}b \right) = 0$. Now

\begin{eqnarray*}
0 & = & \beta \left(\sum_{\substack{1 \leq j \leq 4n \\ \text{$j$ odd}}} a^{j}b \right) = \sum_{\substack{1 \leq k \leq 4n \\ \text{$k$ odd}}} \left( \sum_{\substack{1 \leq i \leq 4n \\ \text{$i$ even}}}\lambda_{i}\right) a^{k}b + \sum_{\substack{1 \leq k \leq 4n \\ \text{$k$ even}}} \left( \sum_{\substack{1 \leq i \leq 4n \\ \text{$i$ odd}}}\lambda_{i}\right) a^{k}b.
\end{eqnarray*}
So $\sum_{\substack{1 \leq i \leq 4n \\ \text{$i$ even}}}\lambda_{i} = 0$ and $\sum_{\substack{1 \leq i \leq 4n \\ \text{$i$ odd}}}\lambda_{i} = 0$ so that $\beta \in \Delta'(\langle a^{2} \rangle)$. Therefore, $\tilde{f}(a) = \beta b$, where $\beta \in \Delta'(\langle a^{2} \rangle)$.\vspace{6pt}

\textbf{Case 2:} $n$ is odd. Then $n = 2m-1$ for some $m \in \mathbb{N}$. So again from (\ref{eq 5.3}), we get that \begin{eqnarray*}\tilde{f}(a^{2n}) & = & 2n \alpha a^{2n-1} + \beta \left(\sum_{k=1}^{n} a^{4(m+k)-3}b + \sum_{k=1}^{n} a^{4k-3}b \right) = 2n \alpha a^{2n-1} + 2 \beta \left(\sum_{k=1}^{n} a^{4k-3}b \right).
\end{eqnarray*}
So $\tilde{f}(a^{2n}) = 0$ if and only if $\alpha = 0$ and $\beta \left(\sum_{k=1}^{n} a^{4k-3}b \right) = 0$. Any even number is of the form $4k-2$ or $4k-4$ for $k \in \mathbb{Z}$ and any odd number is of the form $4k-1$ or $4k-3$ for $k \in \mathbb{Z}$. So, we have the following:
\begin{equation*}
\begin{aligned}
0  &=  \beta \left(\sum_{k=1}^{n} a^{4k-3}b \right) =  \left(\sum_{\substack{1 \leq i \leq 4n \\ \text{$i$ even}}} \lambda_{i} a^{i}\right)\left(\sum_{k=1}^{n} a^{4k-3}b \right) + \left(\sum_{\substack{1 \leq i \leq 4n \\ \text{$i$ odd}}} \lambda_{i} a^{i}\right)\left(\sum_{k=1}^{n} a^{4k-3}b \right) \\ 
  &=  \sum_{t=1}^{n} \left(\sum_{j=1}^{n} \lambda_{4j} \right)a^{4t-3}b + \sum_{t=1}^{n} \left(\sum_{j=1}^{n} \lambda_{4j-2} \right)a^{4t-1}b + \sum_{t=1}^{n} \left(\sum_{j=1}^{n} \lambda_{4j-3} \right)a^{4t-2}b \\ &\quad +  \sum_{t=1}^{n} \left(\sum_{j=1}^{n} \lambda_{4j-1} \right)a^{4t}b.
\end{aligned}
\end{equation*}
So $\sum_{j=1}^{n} \lambda_{4j} = 0, \hspace{0.1cm} \sum_{j=1}^{n} \lambda_{4j-2} = 0, \hspace{0.1cm} \sum_{j=1}^{n} \lambda_{4j-3} = 0, \hspace{0.1cm} \sum_{j=1}^{n} \lambda_{4j-1} = 0$. Therefore, $f(a) = \beta b$, where $\beta = \sum_{i=1}^{4n} \lambda_{i}a^{i} \in \Delta'(\langle a^{2} \rangle) \cap \Delta'(\langle a^{4} \rangle)$ together with the conditions that $\sum_{j=1}^{n} \lambda_{4j-3} = 0$ and $\sum_{j=1}^{n} \lambda_{4j-1} = 0$. 
Now we come out of the two cases and proceed for general proof.
$\tilde{f}(b) = 0$ if and only if $f(b) \in \bar{C}(b)$. Also, using $\tilde{f}(a^{2n}) = 0$, $\tilde{f}(ab) = \beta + af(b)$ and (\ref{eq 5.4}), we get that $\tilde{f}(a^{2n+1}bab) = 0 \Leftrightarrow \tilde{f}(ab) \in \bar{C}(ab)$. Since $f(b) \in \bar{C}(b)$ and by \th\ref{lemma 5.2} (i), $\mathcal{\bar{B}}(b)$ is an $\FF $-basis of $\bar{C}(b)$, so 

\begin{equation*}
\begin{aligned}
f(b) & = \sum_{s=1}^{n-1} \mu_{s}(a^{2s} - a^{-2s}) + \sum_{s=1}^{n-1} \nu_{s} (a^{2s} - a^{-2s})b + \sum_{t=-\lfloor \frac{n}{2} \rfloor}^{\lfloor \frac{n}{2} \rfloor -1} \delta_{t}(a^{2t+1} - a^{2n-(2t+1)}) \\ &\quad + \sum_{t=- \lfloor \frac{n}{2} \rfloor}^{\lfloor \frac{n}{2} \rfloor -1} \gamma_{t}(a^{2t+1} - a^{2n-(2t+1)})b.
\end{aligned}
\end{equation*} and since $\tilde{f}(ab) \in \bar{C}(ab)$, and by \th\ref{lemma 5.2} (ii), $\mathcal{\bar{B}}(ab)$ is an $\FF $-basis of $\bar{C}(ab)$, so 

\begin{equation*}
\begin{aligned}
\tilde{f}(ab) & = \sum_{s=1}^{n-1} \mu_{s} ' (a^{2s} - a^{-2s}) + \sum_{s=1}^{n-1} \nu_{s} ' a(a^{2s} - a^{-2s})b + \sum_{t=- \lfloor \frac{n}{2} \rfloor}^{\lfloor \frac{n}{2} \rfloor -1} \delta_{t} ' (a^{2t+1} - a^{2n-(2t+1)}) \\ &\quad + \sum_{t=- \lfloor \frac{n}{2} \rfloor}^{\lfloor \frac{n}{2} \rfloor-1} \gamma_{t} ' a(a^{2t+1} - a^{2n-(2t+1)})b.
\end{aligned}
\end{equation*}

for some $\mu_{s}, \nu_{s}, \delta_{t}, \gamma_{t}, \mu_{s}', \nu_{s}', \delta_{t}', \gamma_{t}' \in \FF $ ($1 \leq s \leq n-1$ and $- \lfloor \frac{n}{2} \rfloor \leq t \leq  \lfloor \frac{n}{2} \rfloor - 1$). 
Therefore, \begin{equation*}
\begin{aligned}
\beta & = \tilde{f}(ab) - af(b) 
\\ & = \sum_{s=1}^{n-1} \mu_{s} ' (a^{2s} - a^{-2s}) + \sum_{s=1}^{n-1} \nu_{s} ' a(a^{2s} - a^{-2s})b + \sum_{t=- \lfloor \frac{n}{2} \rfloor}^{\lfloor \frac{n}{2} \rfloor-1} \delta_{t} ' (a^{2t+1} - a^{2n-(2t+1)}) \\ &\quad  + \sum_{t=- \lfloor \frac{n}{2} \rfloor}^{\lfloor \frac{n}{2} \rfloor-1} \gamma_{t} ' a(a^{2t+1} - a^{2n-(2t+1)})b - \sum_{s=1}^{n-1} \mu_{s}a(a^{2s} - a^{-2s}) - \sum_{s=1}^{n-1} \nu_{s} a(a^{2s} - a^{-2s})b \\ &\quad - \sum_{t=- \lfloor \frac{n}{2} \rfloor}^{\lfloor \frac{n}{2} \rfloor-1} \delta_{t}a(a^{2t+1} - a^{2n-(2t+1)}) - \sum_{t=- \lfloor \frac{n}{2} \rfloor}^{\lfloor \frac{n}{2} \rfloor -1} \gamma_{t}a(a^{2t+1} - a^{2n-(2t+1)})b.
\end{aligned}
\end{equation*} Since $\beta \in \FF \langle a \rangle$, so $\nu_{s}' = \nu_{s}, \hspace{0.1cm} \forall \hspace{0.1cm} s \in \{1, 2, ..., n-1\}$ and $\gamma_{t}' = \gamma_{t}, \hspace{0.1cm} \forall \hspace{0.1cm} t \in \{-\lfloor \frac{n}{2} \rfloor, ..., \lfloor \frac{n}{2} \rfloor -1\}$. Therefore,
\begin{equation*}
\begin{aligned}
\beta & = \sum_{s=1}^{n-1} \mu_{s} ' (a^{2s} - a^{-2s}) + \sum_{t=- \lfloor \frac{n}{2} \rfloor}^{\lfloor \frac{n}{2} \rfloor -1} \delta_{t} ' (a^{2t+1} - a^{2n-(2t+1)}) - \sum_{s=1}^{n-1} \mu_{s}a(a^{2s} - a^{-2s}) \\ &\quad - \sum_{t=- \lfloor \frac{n}{2} \rfloor}^{\lfloor \frac{n}{2} \rfloor -1} \delta_{t}a(a^{2t+1} - a^{2n-(2t+1)}).
\end{aligned}
\end{equation*}
Therefore, $f(a) = \beta b$ where $\beta$ is as found above. Now put \begin{equation*}
\begin{aligned}\mathcal{B} & = \{d_{(\bar{a}, \bar{b})} \mid (\bar{a}, \bar{b}) \in \{(a^{2s} - a^{-2s})b,0), ~~ ((a^{2t+1} - a^{2n-(2t+1)})b,0), ~~ (a(a^{2s} - a^{-2s})b, \\ &\quad (a^{2s}-a^{-2s})), ~~ (a(a^{2t+1} - a^{2n-(2t+1)})b, (a^{2t+1} - a^{2n-(2t+1)})b), ~~ (0, (a^{2s} - a^{-2s})b), \\ &\quad (0, (a^{2t+1} - a^{2n-(2t+1)})b) \mid 1 \leq s \leq n-1, -\lfloor \frac{n}{2} \rfloor \leq t \leq \lfloor \frac{n}{2} \rfloor -1\}\}.\end{aligned}
\end{equation*} Then $\mathcal{B}$ is an $\FF $-basis of $\der_{\FF }(\FF (SD_{8n}))$. As $|\mathcal{B}| = 3(n-1) + 6 \lfloor \frac{n}{2} \rfloor = \begin{cases} 
3(2n-1) & \text{if $n$ is even}\\
6(n-1) & \text{if $n$ is odd} \end{cases}$, therefore, the dimension of $\der(\FF (DS_{8n}))$ over $\FF $ is $3(2n-1)$ if $n$ is even and $6(n-1)$ if $n$ is odd.\vspace{10pt}

\textbf{(ii)} $\ch(\FF ) = p$ with $\gcd(n,p)=p$: Then $2n \alpha a^{2n-1} = 0$. So (\ref{eq 5.3}) becomes $\tilde{f}(a^{2n}) \\ = \beta \left(\sum_{\substack{1 \leq i \leq 2n \\ \text{$i$ even}}} a^{2n+2i-1}b + \sum_{\substack{1 \leq i \leq 2n \\ \text{$i$ odd}}} a^{2i-1}b \right)$. Now proceeding as in (i), it can be shown that if $n$ is even, $\tilde{f}(a^{2n})=0$ if and only if $\beta \in \Delta'(\langle a^{2} \rangle)$, and if $n$ is odd, $\tilde{f}(a^{2n})=0$ if and only if $\beta = \sum_{i=1}^{4n} \lambda_{i}a^{i} \in \Delta'(\langle a^{2} \rangle) \cap \Delta'(\langle a^{4} \rangle)$ together with the conditions that $\sum_{j=1}^{n} \lambda_{4j-3} = 0$ and $\sum_{j=1}^{n} \lambda_{4j-1} = 0$.
So $f(a) = \alpha + \beta b$, where $\beta$ satisfies the above conditions.
As $\alpha b + \beta = \tilde{f}(ab)-af(b)$, so using the expressions for $f(b)$ and $f(ab)$, we will get that \\ $\alpha = \sum_{s=1}^{n-1} \nu_{s} ' a(a^{2s} - a^{-2s}) + \sum_{t=- \lfloor \frac{n}{2} \rfloor}^{\lfloor \frac{n}{2} \rfloor-1} \gamma_{t} ' a(a^{2t+1} - a^{2n-(2t+1)}) - \sum_{s=1}^{n-1} \nu_{s} a(a^{2s} - a^{-2s}) \\ ~~~~~~ - \sum_{t=- \lfloor \frac{n}{2} \rfloor}^{\lfloor \frac{n}{2} \rfloor -1} \gamma_{t}a(a^{2t+1} - a^{2n-(2t+1)})$ and

\noindent  $\beta = \sum_{s=1}^{n-1} \mu_{s} ' (a^{2s} - a^{-2s}) + \sum_{t=- \lfloor \frac{n}{2} \rfloor}^{\lfloor \frac{n}{2} \rfloor -1} \delta_{t} ' (a^{2t+1} - a^{2n-(2t+1)}) - \sum_{s=1}^{n-1} \mu_{s}a(a^{2s} - a^{-2s}) \\ ~~~~~~ - \sum_{t=- \lfloor \frac{n}{2} \rfloor}^{\lfloor \frac{n}{2} \rfloor -1} \delta_{t}a(a^{2t+1} - a^{2n-(2t+1)})$.

Therefore, $f(a) = \alpha + \beta b$, where $\alpha$ and $\beta$ are as determined above. Put \begin{equation*}
\begin{aligned}\mathcal{B}' & = \{d_{(\bar{a}, \bar{b})} \mid (\bar{a}, \bar{b}) \in \{(a(a^{2s} - a^{-2s}),0), (a(a^{2t+1} - a^{2n-(2t+1)}),0), (a(a^{2s} - a^{-2s}), \\ &\quad (a^{2s} - a^{-2s})b),  (a(a^{2t+1} - a^{2n-(2t+1)}), (a^{2t+1} - a^{2n-(2t+1)})b), ((a^{2s} - a^{-2s})b,0), \\ &\quad ((a^{2t+1} - a^{2n-(2t+1)})b,0),  (a(a^{2s} - a^{-2s})b, (a^{2s} - a^{-2s})), (a(a^{2t+1} - a^{2n-(2t+1)})b, \\ &\quad (a^{2t+1} - a^{2n-(2t+1)}))  \mid 1 \leq s \leq n-1, -\lfloor \frac{n}{2} \rfloor  \leq t \leq \lfloor \frac{n}{2} \rfloor -1\}\}.
\end{aligned}
\end{equation*} Then $\mathcal{B}'$ is an $\FF $-basis of $\der_{\FF }(\FF (SD_{8n}))$. Since $|\mathcal{B}'| = 4(n-1) + 8 \lfloor \frac{n}{2} \rfloor$, therefore, the dimension of $\der_{\FF }(\FF (SD_{8n}))$ over $\FF $ is $4(2n-1)$ if $n$ is even and $8(n-1)$ if $n$ is odd. Hence proved.
\end{proof}

\begin{corollary}\th\label{corollary 5.4}
Let $\FF $ be an algebraic extension of a prime field such that $\ch(\FF )$ is either $0$ or an odd rational prime $p$ relatively prime to $n$. Then the dimension of $\der(\FF (SD_{8n}))$ over $\FF $ is $3(2n-1)$ if $n$ is even and $6(n-1)$ if $n$ is odd, and a basis $\mathcal{B}$ as given in part (i) of \th\ref{theorem 5.3}.
\end{corollary}
\begin{proof}
Follows from \th\ref{theorem 2.12} and \th\ref{theorem 5.3} (i).
\end{proof}

\begin{theorem}\th\label{theorem 5.5}
Let $\FF $ be a field. Then the dimension of $\der_{\inn}(\FF (SD_{8n}))$ over $\FF $ is $3 (2n-1)$ if $n$ is even and $6(n-1)$ if $n$ is odd.

\begin{enumerate}
\item[(i)] When $n$ is even, a basis for $\der_{\inn}(\FF (SD_{8n}))$ is 
\begin{equation*}
\begin{aligned}
\mathcal{B}_{e} & = \{d_{g} \mid g \in \{a^{2k} \mid 1 \leq k \leq n-1\} \cup \{a^{2k+1} \mid - \frac{n}{2} \leq k \leq \frac{n}{2}-1\} \\ &\quad \cup \{a^{2k}b, a^{2k-1}b \mid 1 \leq k \leq 2n-1\}\}.
\end{aligned}
\end{equation*}

\item[(ii)] When $n$ is odd, a basis for $\der_{\inn}(\FF (SD_{8n}))$ is 
\begin{equation*}
\begin{aligned}
\mathcal{B}_{o} & = \{d_{g} \mid g \in \{a^{2k} \mid 1 \leq k \leq n-1\} \cup \{a^{2k+1} \mid - \frac{n-1}{2} \leq k \leq \frac{n-1}{2}-1\} \\ &\quad \cup \{a^{4k}b, a^{4k+1}b, a^{4k+2}b, a^{4k+3}b \mid 1 \leq k \leq n-1\}\}.
\end{aligned}
\end{equation*}
\end{enumerate}
\end{theorem}

\begin{proof}
Follows from \th\ref{theorem 2.15} and \th\ref{lemma 5.1}.
\end{proof}

\begin{corollary}\th\label{corollary 5.6}
Let $\FF $ be an algebraic extension of a prime field and $p$ be an odd rational prime.
\begin{enumerate}
\item[(i)] If $\ch(\FF )$ is either $0$ or $p$ with $\gcd(n,p)=1$, then all derivations of $\FF (SD_{8n})$ are inner, that is,  $\der(\FF (SD_{8n})) = \der_{\inn}(\FF (SD_{8n}))$. In other words, $\FF (SD_{8n})$ has no non-zero outer derivations.

\item[(ii)] If $\ch(\FF ) = p$ with $\gcd(n,p) \neq 1$, then $\der_{\inn}(\FF (SD_{8n})) \subsetneq \der(\FF (SD_{8n}))$. In other words, $\FF (SD_{8n})$ has non-zero outer derivations.
\end{enumerate}
\end{corollary}

\begin{proof}
Follows from \th\ref{theorem 5.3} (ii), \th\ref{corollary 5.4} and \th\ref{theorem 5.5}.
\end{proof}\vspace{12pt}

\noindent \textbf{Acknowledgements}\vspace{8pt}

\noindent The first author receives a monthly institute fellowship from her institute Indian Institute of Technology Delhi for carrying out her doctoral research. The second author is the ConsenSys Blockchain chair professor. He thanks ConsenSys AG for that privilege.\vspace{10pt}

\noindent \textbf{Declaration of Interest Statement}\vspace{8pt}

\noindent The authors report there are no competing interests to declare.

\bibliographystyle{plain}
\end{document}